\documentclass[11pt, a4paper, english, reqno]{amsart}
\usepackage{amsmath} 
\usepackage{amsthm} 
\usepackage{thmtools, thm-restate}
\usepackage{amssymb} 
\usepackage[dvipsnames]{xcolor}
\usepackage{amscd} 
\usepackage{mathtools}
\usepackage{subcaption}
\usepackage{array} 
\usepackage[linktocpage]{hyperref} 
\usepackage{caption} %
\usepackage{graphics,graphicx} 

\usepackage[shortlabels]{enumitem} 
\usepackage[cal=cm,scr=euler]{mathalfa}

\newcommand{\nocontentsline}[3]{}
\newcommand{\tocless}[2]{\bgroup\let\addcontentsline=\nocontentsline#1{#2}\egroup}

\usepackage[margin=1.20in]{geometry}
\usepackage{verbatim}
\usepackage{xcolor}
\usepackage{microtype}
\usepackage{mlmodern}

\hypersetup{
    colorlinks = true,
    linkbordercolor = {white},
    linkcolor = {NavyBlue},
    anchorcolor = {black},
    citecolor = {NavyBlue},
    filecolor = {cyan},
    menucolor = {red},
    runcolor = {cyan},
    urlcolor = {black}
}

\newcommand\CC{{\mathbb C}}
\newcommand\NN{{\mathbb N}}

\newcommand\RR{{\mathbb R}}

\newcommand{\dlb}{[\![}
\newcommand{\drb}{]\!]}
\renewcommand{\H}{\mathrm{H}}
\renewcommand{\tilde}{\widetilde}
\renewcommand{\L}{\mathrm{L}}
\renewcommand{\S}{\mathrm{S}}
\newcommand{\M}{\mathrm{M}}
\newcommand{\supp}{\operatorname{supp}}
\newcommand{\msupp}{\operatorname{msupp}}
\renewcommand{\Im}{\operatorname{Im}}

\DeclareMathOperator{\defect}{def}

\theoremstyle{plain}
\newtheorem{teo}{Theorem}[section]
\newtheorem{theorem}[teo]{Theorem}

\newtheorem{lemma}[teo]{Lemma}
\newtheorem{conj}[teo]{Conjecture}
\newtheorem{question}[teo]{Question}
\newtheorem{proposition}[teo]{Proposition}

\theoremstyle{definition}

\newtheorem{example}[teo]{Example}

\newtheorem{remark}[teo]{Remark}

\newcommand{\ehr}{E}

\DeclareMathOperator{\ULC}{ULC}

\numberwithin{equation}{section}

\title[Preservation of inequalities under Hadamard products]{Preservation of inequalities under\\ Hadamard products}

\author{Petter Br\"and\'en}

\address{
  Department of Mathematics, KTH Royal Institute of Technology, Stockholm, Sweden
}
\email{pbranden@kth.se}

\author{Luis Ferroni}
\address{Dipartimento di Matematica, Universit\`a di Pisa, Pisa, Italy
}
\email{luis.ferroni@unipi.it}

\author{Katharina Jochemko}

\address{
  Department of Mathematics, KTH Royal Institute of Technology, Stockholm, Sweden
}
\email{jochemko@kth.se}


\allowdisplaybreaks

\begin{document}

\begin{abstract}
    Wagner (1992) proved that the Hadamard product of two P\'olya frequency sequences that are interpolated by polynomials is again a P\'olya frequency sequence. We study whether related combinatorial properties are preserved under Hadamard products. In particular, we show that ultra log-concavity, $\gamma$-positivity, and interlacing symmetric decompositions are preserved. Furthermore, we disprove a conjecture by Fischer and Kubitzke (2014) concerning the real-rootedness of Hadamard powers.
\end{abstract}

\maketitle

\section{Introduction}\label{sec:intro}
\thispagestyle{empty}
\noindent Given two formal power series $a,b\in \mathbb{R}\dlb x \drb$,
    \begin{equation}\label{eq:two-power-series}
        a(x) = \sum_{j\geq 0} a_j\, x^j\qquad \mbox{and} \qquad 
        b(x) = \sum_{j\geq 0} b_j\, x^j,
    \end{equation}
the Hadamard product of $a$ and $b$ is the formal power series $a*b\in \mathbb{R}\dlb x\drb$ defined by
    \begin{equation}{\label{eq:hadamard-product}}
        (a*b)(x) := \sum_{j\geq 0} a_jb_j\, x^j.
    \end{equation}
    Formal power series are ubiquitous in enumerative combinatorics and related areas, where the Hadamard product of generating functions often describes fundamental operations of the structures they are enumerating. A prime example is the Ehrhart series of a lattice polytope whose coefficients are given by the number of lattice points inside the dilates of the polytope by nonnegative integer factors; see, e.g.,~\cite{BR}. The Hadamard product of two Ehrhart series then equals the Ehrhart series of the Cartesian product of the corresponding lattice polytopes. More generally, if $a$ and $b$ are the Hilbert series of standard graded algebras then $a * b$ is the Hilbert series of their Segre product; see, e.g., ~\cite{fischer-kubitzke}.

    In many instances in combinatorics, the coefficients of the generating functions are interpolated by a polynomial, i.e., there exists $p\in \mathbb{R}[x]$ such that $a_j=p(j)$ for $j\geq 0$. This is also the case for Ehrhart series: by a result of Ehrhart~\cite{Ehrhart} the number of lattice points in the $j$-th dilate of a lattice polytope $P$ is given by a polynomial $\ehr _P(j)$, called the Ehrhart polynomial of the polytope, whose degree equals the dimension of~$P$. 

    For any polynomial $p\in \mathbb{R}[x]$, one may define the polynomial $\mathscr{W}(p)\in \mathbb{R}[x]$ as the numerator of the rational generating function
    \begin{equation}\label{eq:w-operator}
        \sum_{j\geq 0} p(j)\, x^j = \frac{\mathscr{W}(p)(x)}{(1-x)^{\deg p+1}} \, .
    \end{equation}
    Then $\mathscr{W}(p)$ is always a polynomial of degree at most $\deg p$. The purpose of the present paper is to study the preservation of combinatorially interesting inequalities satisfied by the coefficients of the numerator $\mathscr{W}(p)$ under the Hadamard product of generating functions.

    A fundamental result by Wagner~\cite{wagner} states that the Hadamard product preserves the property of being a P\'olya frequency sequence which is interpolated by a polynomial, that is, if $(a_j)_{j \geq 0}$ and $(b_j)_{j\geq 0}$ are P\'olya frequency sequences interpolated by polynomials then so is $(a_jb_j)_{j\geq 0}$. This can be expressed in terms of the numerator polynomial in the following way.

    \begin{theorem}[{\cite[Theorem~0.2]{wagner}}]\label{thm:wagner-main}
    If $\mathscr{W}(p_1)$ and $\mathscr{W}(p_2)$ have only nonpositive real zeros, then so does $\mathscr{W}(p_1p_2)$.
\end{theorem}

In the present paper we prove results in the same spirit of Theorem~\ref{thm:wagner-main}, with a particular focus on properties related to unimodality and log-concavity of the numerator polynomial defined further below. The recent development of the theory of Lorentzian polynomials by Br\"and\'en and Huh~\cite{branden-huh} and Anari \emph{et al.}~\cite{Anari} provides powerful new tools to approach log-concavity questions whose full potential still needs to be explored. We contribute to this development by applying these tools to Hadamard products. 

In the setting of Ehrhart theory, the numerator polynomial of the Ehrhart series of a lattice polytope $P$ is called the $h^\ast$-polynomial, denoted $h^\ast _P (x)$:
\[
\sum _{j\geq 0}\ehr _P (j)\,x^j \ = \ \frac{h^\ast _P (x)}{(1-x)^{\dim P +1}} \, .
\]
By a fundamental result due to Stanley~\cite{stanley-h-star}, the $h^\ast$-polynomial has always only nonnegative integer coefficients. Further inequalities satisfied by the coefficients have been intensively studied. In particular, in the focus of current ongoing research are questions about the unimodality of the coefficient sequence of the $h^\ast$-polynomial, with a hierarchy of open conjectures; see~\cite{stanley-unimodality,hibi-ohsugi,braun,adiprasito-papadakis-petrotou-steinmeyer,ferroni-higashitani} and references therein.

The following inequalities are central for the present work. A polynomial with nonnegative coefficients $\smash{a(x)=\sum _{j=0}^d a_jx^j}$ is called \emph{unimodal} if there exists an index~$k$ for which $ a_0\leq \cdots \leq a_k \geq \cdots \geq a_d$. It is called \emph{log-concave} whenever $\smash{a_j^2\geq a_{j-1}a_{j+1}}$ for every $1\leq j \leq d -1$. We say that $a(x)$ has \emph{no internal zeros} if there are no indices $0\leq i < j < k \leq d$ such that $a_{j}=0$ and $a_{i}a_{k}\neq 0$. From the definitions it follows that any log-concave polynomial with no internal zeros is unimodal. For $m\geq d$, we say that $a(x)$ is \emph{ultra log-concave of order $m$} if the sequence
\[\left( \frac{a_j}{\binom{m}{j}}\right)_{j=0}^d\]
is log-concave. Following Liggett \cite{liggett}, we denote the set of ultra log-concave polynomials of order $m$ with no internal zeros by $\ULC(m)$. Notice that the Newton inequalities (see, e.g., \cite[Theorem~2]{stanley-unimodality}) imply that any polynomial $a(x)$ having only nonpositive real zeros is contained in $\ULC(\deg a)$, while the converse is easily seen to be false. In summary, if $a(x)$ is any polynomial with nonnegative coefficients, then real-rootedness (i.e., having only real zeros) implies ultra log-concavity of order $\deg a$ with no internal zeros, which in turn implies unimodality. Wagner's Theorem~\ref{thm:wagner-main} shows preservation of the most restrictive of these properties, while from \cite[Example~3.5]{ferroni-higashitani} it is known that unimodality is in general not preserved under the Hadamard product. In the present work, we shall investigate the middle grounds. 

Throughout this paper, we shall always assume that polynomials of the form $\mathscr{W}(p)$ have nonnegative coefficients, a property that is also preserved by the Hadamard product as shown by Fischer and Juhnke-Kubitzke~\cite[Proposition~5.1(i)]{fischer-kubitzke}; we recover this result in Lemma~\ref{lem:magiclemma} below.

The following is our first main result, and it provides a counterpart of Theorem~\ref{thm:wagner-main} for ultra log-concavity.

\begin{restatable}{theorem}{thmulc}\label{thm:ultra-log-concavity}
    If $\mathscr{W}(p_1)\in \ULC(\deg p_1)$ and $\mathscr{W}(p_2)\in \ULC(\deg p_2)$, then $\mathscr{W}(p_1p_2)\in \ULC(\deg p_1+\deg p_2)$.
\end{restatable}

For the proof we rely crucially on the theory of Lorentzian polynomials developed in~\cite{branden-huh}; we briefly review the necessary definitions and properties in Section~\ref{sec:ultra-log-concavity}. Moreover, in a sense, our result \emph{refines} Wagner's theorem, because our proof simultaneously also yields an independent proof of Theorem~\ref{thm:wagner-main}. We also show that the Hadamard product preserves the property of having no internal zeros; see Proposition~\ref{prop:no-int-zeros} below. Within Ehrhart theory, this is related to results by Hofscheier, Katth\"an and Nill~\cite{spanning} who showed that spanning lattice polytopes, which are closed under Cartesian products, have $h^\ast$-polynomials with no internal zeros. Further, in Example~\ref{ex:no-int-zeros-nec}, we provide polynomials that show that the preservation of ultra log-concavity may fail if one removes the assumption on no internal zeros.

Many polynomials arising in combinatorics and related fields are symmetric (also called palindromic), for example, $h^\ast$-polynomials of reflexive and more general Gorenstein lattice polytopes or $Z$-polynomials of matroids~\cite{Zpolynomial}. We investigate preservation of $\gamma$-positivity under the Hadamard product, a property of symmetric polynomials that is stronger than unimodality. For any polynomial $a(x)$ of degree $s\leq d$, we define $\mathcal{I}_d(a)(x) = x^d a(1/x)$. We say that $a$ is \emph{symmetric} with center of symmetry $d/2$ whenever $\mathcal{I}_d(a) = a$. Any such polynomial can be uniquely written in the basis $\{x^i(1+x)^{d-2i}\}_{i\geq 0}$. In other words, we can write
    \[ a(x) = \sum_{i=0}^{\lfloor d/2 \rfloor} \gamma_i\, x^i(1+x)^{d-2i}.\]
We say that $a$ is \emph{$\gamma$-positive} if $\gamma_i\geq 0$ for all $i$. A result noted independently by Gal \cite{gal} and Br\"and\'en \cite{branden-2004} states that if $a(x)$ is symmetric and real-rooted, then it is $\gamma$-positive. Stronger, they noted that $a(x)$ is real-rooted if and only if the \emph{$\gamma$-polynomial} $$\gamma _a (x) := \sum _{i=0}^{\lfloor {d}/2\rfloor}\gamma_i x^i$$ is real-rooted. For a comprehensive survey on $\gamma$-positivity questions in combinatorics we refer to~\cite{athanasiadis-survey}.

For a polynomial $p$ of degree $d$ with $\mathcal{I}_s \mathscr{W}(p)= \mathscr{W}(p)$, let us define its \emph{defect}, denoted $\defect(p)$, as $d - s$. In Ehrhart theoretic terms, if $\ehr _P$ is the Ehrhart polynomial of a Gorenstein polytope then $\defect(\ehr _P)$ equals the \emph{degree} of the lattice polytope $P$ plus~$1$. As our second main result we show that $\gamma$-positivity is preserved under Hadamard products,  provided both factors have the same defect, thus offering a further analog of Theorem~\ref{thm:wagner-main} for $\gamma$-positivity.

\begin{theorem}\label{thm:gamma-positivity}
    Let $p,q\in\mathbb{R}[x]$ be polynomials such that $\mathscr{W}(p)$ and $\mathscr{W}(q)$ are symmetric and $\defect(p) = \defect(q)$. Then, $\mathscr{W}(pq)$ is symmetric, and $\defect(pq)= \defect(p) = \defect(q)$. Furthermore, if $\mathscr{W}(p)$ and $\mathscr{W}(q)$ are $\gamma$-positive, then $\mathscr{W}(pq)$ is $\gamma$-positive, too.
\end{theorem}

The first part in the above statement, preservation of symmetry under the Hadamard product, was noted by Fischer and Kubitzke in \cite[Proposition~5.1(ii)]{fischer-kubitzke}. Indeed, the condition on the defects of $\mathscr{W}(p_1)$ and $\mathscr{W}(p_2)$ is necessary to guarantee that $\mathscr{W}(p_1p_2)$ is symmetric in the prior statement. Furthermore, we give an analog of the results of Gal \cite{gal} and Br\"and\'en~\cite{branden-2004} for ultra log-concave polynomials: in Theorem~\ref{thm:gamma-ulc} we show that if the $\gamma$-polynomial of a symmetric polynomial is ultra log-concave then also the polynomial itself is ultra log-concave. Example~\ref{ex:noconversegammaulc} shows that the converse is not true in general.

Our third main contribution concerns the preservation of properties for symmetric decompositions under Hadamard products. For every polynomial $h$ of degree at most $d$, there exist unique symmetric polynomials $a$ and $b$ such that $h(x)=a(x)+x\,b(x)$, $a=\mathcal{I}_d(a)$, and $b=\mathcal{I}_{d-1}(b)$. The pair $(a,b)$ is called the \emph{symmetric decomposition} of $h$. In Ehrhart theory, symmetric decompositions were successfully used by Stapledon~\cite{stapledon} to give combinatorial proofs for linear inequalities satisfied by the coefficients of $h^\ast$-polynomials due to Stanley~\cite{StanleyHilbertfunction} and Hibi~\cite{Hibi}.  We say that $h$ has a \emph{nonnegative symmetric decomposition} if both $a$ and $b$ have only nonnegative coefficients. Similarly, we say that the symmetric decomposition is $\gamma$-positive if both $a$ and $b$ are $\gamma$-positive. We prove the following.

\begin{theorem}\label{thm:nonnegativesymdec}
    If $\mathscr{W}(p)$ and $\mathscr{W}(q)$ have nonnegative symmetric decompositions, then so does $\mathscr{W}(pq)$. Furthermore, if the symmetric decompositions of $\mathscr{W}(p)$ and $\mathscr{W}(q)$ are $\gamma$-positive, then so is the symmetric decomposition of $\mathscr{W}(pq)$.
\end{theorem}

A property that recently has attracted attention from the community is that of interlacing symmetric decompositions; see~\cite{branden-solus,athanasiadis-tzanaki,jochemko} and references therein. Let $a$ and $b$ be real-rooted polynomials, with zeros $s_k\leq \cdots \leq s_1$ and $t_m\leq \cdots \leq t_1$, respectively. By definition, $b$ \emph{interlaces} $a$, written $b\preceq a$, if
    \[ \cdots \leq t_2 \leq s_2\leq t_1 \leq s_1.\]

We say that the symmetric decomposition $(a,b)$ of a polynomial $h$ is \emph{real-rooted} if $a$ and $b$ are real-rooted, and we say that it is \emph{interlacing} if $b\preceq a$. We prove that the latter property is preserved under Hadamard products.

\begin{theorem}\label{thm:interlacingsymdec}
    Let $\mathscr{W}(p)$ and $\mathscr{W}(q)$ have nonnegative, interlacing symmetric decompositions. Then also $\mathscr{W}(pq)$ has a nonnegative, interlacing symmetric decomposition.
\end{theorem}

Furthermore, we show in Example~\ref{example:real-rootedness-sym-fails} that the real-rootedness of the symmetric decompositions is not preserved under Hadamard products. That is, we exhibit two polynomials $\mathscr{W}(p)$ and $\mathscr{W}(q)$ with real-rooted symmetric decompositions such that the polynomial $\mathscr{W}(pq)$ does not have a real-rooted symmetric decomposition. 

Our final contribution concerns Hadamard powers, and a conjecture regarding their behavior for large exponents. Using the operator defined in equation~\eqref{eq:w-operator}, the following statement is a reformulation of a conjecture posed by Fischer and Kubitzke~\cite{fischer-kubitzke}.

\begin{conj}[{\cite[Conjecture~5.2]{fischer-kubitzke}}]\label{conj:fischer-kubitzke}
    If $\mathscr{W}(p)$ is a polynomial with nonnegative coefficients, then for all sufficiently large positive integers $n$, the polynomial $\mathscr{W}(p^n)$ is real-rooted.
\end{conj}

We will provide a counterexample to this conjecture, and even to several weaker variants of the conjecture.

\begin{theorem}\label{thm:counterexample}
    There exists a polynomial $p$ such that $\mathscr{W}(p)$ has nonnegative coefficients and $\mathscr{W}(p^n)$ is not real-rooted for any $n\geq 1$.
\end{theorem}
In fact, there exists an Ehrhart polynomial $p$ such that $\mathscr{W}(p^n)$ is not log-concave for any $n\geq 1$.

\subsection*{Outline}

In Section~\ref{sec:ultra-log-concavity} we show preservation of ultra log-concavity under the Hadamard product  (Theorem~\ref{thm:ultra-log-concavity}), after recalling necessary definitions and properties of Lorentzian polynomials that are employed in the proof. Section~\ref{sec:gammapositivity} is devoted to $\gamma$-positivity. In particular, we prove that the Hadamard product preserves $\gamma$-positivity (Theorem~\ref{thm:gamma-positivity}). Further, we show that any symmetric polynomial having an ultra log-concave $\gamma$-polynomial is necessarily ultra log-concave (Theorem~\ref{thm:gamma-ulc}). In Section~\ref{sec:sym-dec} we focus on symmetric decompositions, give background on the necessary prerequisites developed in~\cite{branden-solus} and prove preservation of nonnegative, $\gamma$-positive as well as interlacing symmetric decompositions under the Hadamard product (Theorems~\ref{thm:nonnegativesymdec} and~\ref{thm:interlacingsymdec}). In Section~\ref{sec:Hadamardpowers} we provide a counterexample to Conjecture~\ref{conj:fischer-kubitzke}. We conclude in Section~\ref{sec:Final remarks} with final remarks and open questions.

\section{Ultra log-concavity}\label{sec:ultra-log-concavity}

\noindent The main goal of this section is to prove Theorem~\ref{thm:ultra-log-concavity}, for which we will employ the theory of Lorentzian polynomials~\cite{branden-huh}. We begin by briefly recalling basic definitions and fundamental properties of Lorentzian polynomials that are needed in the proof.

Consider a polynomial $p\in \mathbb{R}[x_1,\ldots,x_n]$ defined by
    \[ p(x_1,\ldots,x_n) = \sum_{\alpha\in \mathbb{N}^n} c_{\alpha}\, x^{\alpha},\]
where $\mathbb{N}$ denotes the set of nonnegative integers, and $x^{\alpha}:= \prod_{i=1}^n x_i^{\alpha_i}$. The \emph{support} of $p$ is defined by
    \[ \supp p = \{ \alpha\in \mathbb{N}^n : c_{\alpha} \neq 0\}.\]
An \emph{$\mathrm{M}$-convex} set is a set $\mathrm{J}\subseteq \mathbb{N}^n$ satisfying the following exchange axiom: for any $\alpha,\beta\in \mathrm{J}$ and any $i$ such that $\alpha_i>\beta_i$ there is an index $j$ such that $\alpha_j < \beta_j$ and $\alpha-e_i+e_j\in \mathrm{J}$. Following the notation in~\cite{branden-huh}, we denote by $\H_n^d$ the space of homogeneous polynomials of degree $d$ in $n$ variables and by $\M_n^d\subseteq \H_n^d$ the subset of polynomials having an $\M$-convex support. 

A polynomial $p \in \mathbb{R}[x_1,\ldots, x_n]$ is said to be \emph{stable} if $\Im(x_i) > 0$ for all $1\leq i\leq n$ implies that $p(x_1,\ldots, x_n)\neq 0$, i.e., $p$ is non-vanishing when all the variables take values in the upper open complex half-plane. The set of homogeneous stable polynomials of degree $d$ having nonnegative coefficients is denoted by $\S_n^d\subseteq \H_n^d$.

The partial derivative with respect to the variable $x_i$ will be denoted by $\partial_i$. For $d\leq 2$, define $\L_n^d := \S_n^d$, and for $d > 2$, 
    \[ \L_n^d := \left\{ f\in \M_n^d : \partial_i f\in \L_n^{d-1} \text{ for all $1\leq i\leq n$}\right\}.\]
The polynomials in $\L_n^d$ are called \emph{Lorentzian}. In~\cite{branden-huh} the following strict inclusions are proven
\[
\S_n^d \subsetneq \L_n^d \subsetneq \M_n^d \subsetneq \H_{n}^d \, 
\]
for $d>2$. 

In the bivariate case, the sets of Lorentzian polynomials $\L_2^d\subseteq \mathbb{R}[x,y]$ and homogeneous stable polynomials with nonnegative coefficients $\S_2^d\subseteq \mathbb{R}[x,y]$ can be described in a succinct way as follows. For every $p\in \H_2^d$ we have
    \begin{align*}
        p\in \S_2^d  \iff & p(x,1) \text{ is a real-rooted polynomial},\\
        p\in \L_2^d  \iff & p(x,1) \in \ULC(d).
    \end{align*}
    Equivalently, if $p(x)$ is a univariate polynomial with nonnegative coefficients and $\mathcal{H}_d(p)(x,y)=y^dp({x}/{y})\in \H_2^d$ be its homogenization to degree $d$, then $p\in \ULC(d)$ (resp. $p$ is real-rooted) if and only if $\mathcal{H}_d(p)\in \RR[x,y]$ is Lorentzian (resp. stable). See also~\cite[Example~2.26]{branden-huh}.

    Next we discuss operations that preserve stable and Lorentzian polynomials. For an element $\kappa\in \mathbb{N}^n$, define $\mathbb{R}_{\kappa}[x_1,\ldots,x_n]$ to be the space of polynomials in the variables $x_1,\ldots,x_n$ in which the variable $x_i$ appears with exponent at most $\kappa_i$, for each $i=1,\ldots,n$. Consider $\kappa,\lambda\in \mathbb{N}^n$ and a linear map $T:\mathbb{R}_{\kappa}[x_1,\ldots,x_n]\to \mathbb{R}_{\lambda}[x_1,\ldots,x_n]$. We say that $T$ is \emph{homogeneous of degree $\ell$}, where $\ell\in \mathbb{Z}$, if each monomial $x^{\alpha}$ is either mapped to the zero polynomial or $\deg T(x^{\alpha}) = |\alpha| + \ell$. 

\begin{theorem}[{\cite[Theorem~3.4]{branden-huh}}]\label{thm:stable-preserves-lorentzian}
    If $T$ is a homogeneous linear map that preserves stable polynomials and polynomials with nonnegative coefficients, then $T$ preserves Lorentzian polynomials.
\end{theorem}

We will use the following fundamental operations, which preserve stability and the Lorentzian property

\begin{theorem}\label{thm:products-and-changes-of-var}
    Let $p$ and $q$ be polynomials with nonnegative coefficients. 
    \begin{enumerate}[\normalfont (i)]
        \item\label{it:positive-change-var} \cite[Theorem~2.10]{branden-huh} If $p(x_1,\ldots,x_m)$ is Lorentzian (resp. stable) and $A\in \mathbb{R}^{m\times n}$ is a matrix with nonnegative entries, then $p(A\cdot (x_1,\ldots,x_n)^T)$ is Lorentzian (resp. stable).
        \item\label{it:products} \cite[Corollary~2.32]{branden-huh} If $p$ and $q$ are Lorentzian (resp. stable) then so is their pro\-duct~$pq$.
    \end{enumerate}
\end{theorem}
In particular, diagonalization, i.e., setting several variables to be equal, preserves both stability and the Lorentzian property. Furthermore, the coefficients of stable (resp. Lorentzian) polynomials are again stable (resp. Lorentzian) as both of these properties are also preserved under differentiation and setting variables to zero.

Next, we translate the statement of Theorem~\ref{thm:ultra-log-concavity} into the setting of Lorentzian polynomials.


Let $\RR_d[x]$ be the space of polynomials in $\RR[x]$ of degree at most $d$, and let $\RR[x,y]_{d}$ be the linear space of homogeneous bivariate polynomials in $\RR[x,y]$ of degree $d$. The set $\{\binom {x+d-i} d\}_{i=0}^d$ is a basis of $\RR_d[x]$. We define the bijective linear map  $T_d : \RR_d[x]\longrightarrow \RR[x,y]_d$ by 
\[\binom {x+d-i} d \longmapsto x^iy^{d-i}, \ \ \mbox{ for } 0\leq i \leq d,\]
and consider the bilinear map 
$\RR[x,y]_a \times \RR[x,y]_b \longrightarrow \RR[x,y]_{a+b}$, $(p,q)\longmapsto p\bullet q$, defined by 
\[p \bullet q =  T_{a+b}\big( T_a^{-1}(p)\cdot T_b^{-1}(q) \big). \]
Observe that for any polynomial $p(x)$ of degree $d$, $\mathscr{W}(p)=\sum _{i=0}^d h_i x^i$ if and only if $p(x)=\sum _{i=0}^d h_i\binom {x+d-i} d$. Thus, by definition, $T_d (p)=\mathcal{H}_d\left(\mathscr{W}(p)\right)$. Therefore, for any polynomials $p$ and $q$
\begin{equation}\label{eq:homog} \mathcal{H}\left(\mathscr{W}(p)\right)\bullet \mathcal{H}\left(\mathscr{W}(q)\right) \ = \ \mathcal{H}\left(\mathscr{W}(pq)\right),
\end{equation}
where in each of the above factors the homogenization operators $\mathcal{H}$ are applied with respect to the degrees of $p$, $q$, and $pq$ respectively.

Thus, Theorem~\ref{thm:ultra-log-concavity} (resp. Theorem~\ref{thm:wagner-main}) is equivalent to proving that if $p$ and $q$ are Lorentzian (resp. stable with nonnegative coefficients), $p\bullet q$ is Lorentzian (resp. stable with nonnegative coefficients). In order to do that, we begin by first explicitly describing the action of the bilinear map $(p,q)\mapsto p\bullet q$ on monomials.  

\begin{lemma}\label{lemma:comb-identity}
For integers $0\leq k \leq a$ and $0\leq \ell \leq b$, 
\begin{equation}\label{bilin}
    \left(x^ky^{a-k}\right)\bullet \left(x^\ell y^{b-\ell}\right) =  \sum_i \binom {a-k+\ell}{i-k} \cdot \binom {b-\ell+k}{i-\ell} \cdot x^i y^{a+b-i} \, .
\end{equation}
\end{lemma}
\begin{proof}
    We rely on the following combinatorial identity attributed to Nanjundiah, cf. Riordan's book~\cite[Equation~(12),\ p.~16]{riordan}, 
    \[ \binom{m}{p}\binom{n}{q} = \sum_{j} \binom{n+j}{p+q}\binom{m-n+q}{j}\binom{n-m+p}{p-j}. \]
    Specializing $m = x + a - k$, $n = x + b - \ell$, $p = a$, $q = b$, and changing the variable of the summation by $i=a+j-\ell$ we obtain 
        \[
    \binom {x+a-k} a \cdot \binom {x+b-\ell} b = \sum_i \binom {a-k+\ell}{i-k} \cdot \binom {b-\ell+k}{i-\ell} \cdot \binom {x+a+b-i}{a+b}.
    \]
    from which the claim follows.
\end{proof}

Next, for homogeneous polynomials $p$ and $q$ we express $(p \bullet q)(x,y)$ as a coefficient of a homogeneous polynomial in three variables.
\begin{lemma}\label{form}
Suppose $p \in  \RR[x,y]_{a}$ and $q \in  \RR[x,y]_{b}$. Then $(p \bullet q)(x,y)$ is equal to the coefficient of $z^{2b}$ in 
    \begin{equation}\label{eq:poly-r}
        r(x,y,z) := p(x+z,y+z)(x+z)^b(y+z)^b q\left( \frac {xz} {x+z}, \frac {yz} {y+z}\right).
    \end{equation}
\end{lemma}

\begin{proof}
    The polynomial $r(x,y,z)$ depends bilinearly on $p$ and $q$. Thus, it suffices to prove the claim for $p=x^ky^{a-k}$ and $q=x^\ell y^{b-\ell}$. In this case, the calculation shows that $r(x,y,z)$ equals

    \[
    x^\ell y^{b-\ell}z^b (x+z)^{b+k-\ell}(y+z)^{a+\ell-k}\, .
    \]
    The coefficient of $z^{2b}$ in this polynomial equals
    \[
    \sum_{j} \binom {b+k-\ell} {j} \binom {a+\ell-k}{b-j} x^{b+k-j}y^{a-k+j}
    \]
    which agrees with the right hand side of~\eqref{bilin} after substituting $i=b+k-j$. 
\end{proof}

\begin{lemma}\label{transf}
Suppose $p(x_1, x_2,\ldots, x_n)$ is a polynomial of degree at most $m$ in $x_1$, and let 
    \[
    q(x_1, y_1,x_2, x_3,\ldots, x_n) = (x_1+y_1)^m p\left( \frac {x_1y_1}{x_1+y_1}, x_2,\ldots, x_n\right).
    \]
    \begin{enumerate}[\normalfont(i)]
        \item \label{it:stable} If $p$ is stable, then so is $q$.
        \item \label{it:lorentzian} If $p$ is Lorentzian, then so is $q$.
    \end{enumerate}
\end{lemma}

\begin{proof}
    Let $H= \{z \in \CC : \mathrm{Im}(z) >0\}$ be the open upper half-plane of $\CC$. The map 
    \[
    (x,y) \longmapsto \frac {xy}{x+y} = \frac {1}{x^{-1}+y^{-1}}
    \]
    maps $H \times H$ to $H$, from which~\ref{it:stable} follows. 
    The map $p \longmapsto q$ defines a homogeneous map which preserves stability by~\ref{it:stable}, thus the validity of~\ref{it:lorentzian} is a consequence of Theorem~\ref{thm:stable-preserves-lorentzian}. 
\end{proof}

We have now all the ingredients to prove the following main result of this section. 

\begin{theorem}\label{lorpres}
    If $p$ and $q$ are two Lorentzian (resp. homogeneous stable) bivariate polynomials, then $p \bullet q$ is Lorentzian (resp. homogeneous stable). 
\end{theorem}

\begin{proof}
 Since coefficients of Lorentzian (resp. homogeneous stable) polynomials are again Lorentzian (resp. homogeneous stable), by Lemma~\ref{form} it suffices to argue that the polynomial $r(x,y,z)\in \mathbb{R}[x,y,z]$ defined as in equation~\eqref{eq:poly-r} is Lorentzian (resp. homogeneous stable). Since $p(x+z,y+z)$ is Lorentzian (resp. homogeneous stable) by Theorem~\ref{thm:products-and-changes-of-var}, it remains to verify that the factor 
    \[(x+z)^b(y+z)^b q\left(\frac{xz}{x+z},\frac{yz}{y+z}\right)\]
    is Lorentzian (resp. homogeneous stable). This follows by applying Lemma~\ref{transf} twice, introducing two new variables, followed by setting both these equal to $z$.
\end{proof}

\begin{proof}[Proof of Theorems~\ref{thm:wagner-main} and~\ref{thm:ultra-log-concavity}]
    By equation~\eqref{eq:homog}, the preservation of the Lorentzian property (resp. stability) under the bilinear map $-\bullet -$ is equivalent to preservation of ultra log-concavity (resp. real-rootedness) under the Hadamard product. Thus, Theorems~\ref{thm:ultra-log-concavity} and Theorem~\ref{thm:wagner-main} follow immediately from Theorem~\ref{lorpres}. 
\end{proof}

\begin{example}\label{ex:no-int-zeros-nec}
    The condition that the polynomials lack internal zeros is essential in Theorem~\ref{thm:ultra-log-concavity}. Consider the polynomials
    \begin{align*}
        p(x) &= \frac{1}{720}\left(2x^6+24x^5+170x^4+720x^3+1628x^2+1776x+720\right)\\
        q(x) &= \frac{1}{6}\left(x^3+6x^2+11x+6\right).
    \end{align*}
    One may compute that $\mathscr{W}(p) = x^3+1$ and $\mathscr{W}(q) = 1$. Both of these polynomials are ultra log-concave, but $\mathscr{W}(p)$ has internal zeros. However, we have
    \[\mathscr{W}(pq) = x^{6} + 18 \, x^{5} + 45 \, x^{4} + 40 \, x^{3} + 45 \, x^{2} + 18 \, x + 1,\]
    which is not ultra log-concave (in fact, not even unimodal). However, as we will see in Proposition~\ref{prop:no-int-zeros}, the property of having no internal zeros is itself preserved under this product.
\end{example}

\section{Gamma positivity}\label{sec:gammapositivity}

\noindent In this section we consider $\gamma$-polynomials and $\gamma$-positivity of symmetric polynomials. First, we will show that the Hadamard product preserves $\gamma$-positivity (Theorem~\ref{thm:gamma-positivity}). Furthermore, we will prove that any symmetric polynomial having an ultra log-concave $\gamma$-polynomial is itself ultra log-concave. 



\subsection{Hadamard products and \texorpdfstring{$\gamma$}{Gamma}-positivity}

For any polynomial $h(x)$ of degree $s\leq d$, let $\mathcal{I}_d(h)(x) = x^d h(1/x)$. In order to prove Theorem~\ref{thm:gamma-positivity}, we use the next folklore lemma which follows from a more general result on rational generating functions (see, e.g., \cite[Corollary~4.2.4(iii)]{stanley-ec1}). For the sake of completeness, we provide a self-contained proof here.

\begin{lemma}\label{lem:functionaleq}
    Let $p$ be a polynomial of degree $d$, denote $h = \mathscr{W}(p)$, and assume that $\deg h \leq s \leq d$. Then, the following are equivalent.
    \begin{enumerate}[\normalfont(i)]
        \item $\mathcal{I}_s h(x) = h(x)$, 
        \item $(-1)^d p (-(x+d+1-s)) = p(x)$.
    \end{enumerate}
\end{lemma}

\begin{proof}
    Let us write $h(x) = \sum _{i=0}^s h_i x^i$. Then,
    \[
    p(x) = \sum _{i=0}^s h_i \binom{x+d-i}{d} \, .
    \]
    Hence, by applying the ``combinatorial reciprocity'' identity on binomial coefficients \cite[p.~20]{stanley-ec1} we obtain:
    \[
        (-1)^d p(-(x+d+1-s)) = (-1)^d \sum _{i=0}^s h_i \binom{-x+s-i-1}{d}
        = \sum _{i=0}^s h_{s-i} \binom{x+d-i}{d} \, .
    \]
The proof follows since the polynomials $\left\{\binom{x+d-i}{d}\right\}_{i=0}^s$ are linearly independent.
\end{proof}

\begin{remark}\label{remark:dual-lemma}
    There is a `dual' version of the previous lemma for the case $d\leq s \leq 2d$. In that scenario, the conditions (i) and (ii) are equivalent under the assumption that $h_0=h_1=\ldots = h_{s-d-1}=0$. The proof is analogous. \\
    Recall that the vanishing of the coefficients $h_i$ for $s < i \leq d$ is equivalent to $p(-1) = p(-2) = \cdots = p(-(d-s)) = 0$. Dually, if $s>d$, the vanishing of coefficients $h_i$ for $i < s-d$ is equivalent to $p(0) = p(1) = \cdots = p(s-d-1) = 0$.
\end{remark}

Now assume that $p$ is a polynomial of degree $d$, and that $\mathscr{W}(p)$ is symmetric with center of symmetry $s/2$, i.e., $\mathcal{I}_s\mathscr{W}(p) = \mathscr{W}(p)$. We define the \emph{defect} of $p$ as the quantity $\deg (p) - s$. We will denote this number by $\defect(p)$. 

Observe that the defect is a negative number if $d<s\leq 2d$. For example, if $p(x) = \tfrac{1}{6}x(x-1)(x+1)$, one obtains $h(x):=\mathscr{W}(p)(x) = x^2$ and, in the notation above, one has $\mathcal{I}_4h(x) = h(x)$, so that $d = 3$, $s=4$ and $\defect(p) = 3-4 = -1$.

Notice that if $c = \defect(p)$, then Lemma~\ref{lem:functionaleq} (or Remark~\ref{remark:dual-lemma}) implies
    \[p(x) = (-1)^d p(-(x+c+1)) .\]

We are now able to prove our second main result, namely that symmetry and $\gamma$-positivity is preserved under Hadamard products, given that both factors have the same defect.

\begin{proof}[Proof of Theorem~\ref{thm:gamma-positivity}]
    Let us denote $d = \deg p$ and $e = \deg q$. By Lemma~\ref{lem:functionaleq}, since $\mathscr{W}(p)$ and $\mathscr{W}(q)$ are symmetric, we have
    \begin{align*}
        (-1)^{d+e} pq (-(x+c+1))&=(-1)^dp(-(x+c+1))(-1)^e q(-(x+c+1))\\
        &=p(x)q(x),
    \end{align*}
    where $c = \defect(p) = \defect(q)$. Further, $pq(-i)=p(-i)q(-i)=0$ for $1\leq i\leq c$  (or $pq(i)=p(i)q(i)=0$ for $0\leq i\leq -c-1$ in the dual setting when $s>d$ and $c$ is negative; see Remark~\ref{remark:dual-lemma}.) Thus, $\mathscr{W}(pq)$ is again symmetric by Lemma~\ref{lem:functionaleq},  and evidently $\defect(pq) = c$. Regarding $\gamma$-positivity, by bilinearity it suffices to prove the claim for 
    \begin{align*}
     \mathscr{W}(p)&= x^i (x+1)^{s-2i}  &\text{for $i\leq s/2$},\\
     \mathscr{W}(q)&= x^j (x+1)^{\ell-2j}  &\text{for $j\leq \ell/2$},
    \end{align*}
    where $s = d - c$ and $\ell = e - c$. In that case, both $\mathscr{W}(p)$ and $\mathscr{W}(q)$ are real-rooted, and thus by Theorem~\ref{thm:wagner-main}, $\mathscr{W}(pq)$ is real-rooted. Since $\mathscr{W}(pq)$ is symmetric, we have that $\mathscr{W}(pq)$ is $\gamma$-positive.
\end{proof}

\subsection{Ultra log-concavity of \texorpdfstring{$\gamma$}{gamma}-polynomials}

We make a brief digression in our discussion. A current trend in matroid theory is that of proving various log-concavity results for polynomials that arise from Hodge theory. Two families of polynomials that are conjectured to be real-rooted are the Kazhdan--Lusztig and the $Z$-polynomial of a matroid \cite{gedeon-proudfoot-young}. The log-concavity for these two polynomials remains open even in very simple cases, though it is known that the $Z$-polynomial is palindromic and $\gamma$-positive \cite{ferroni-matherne-stevens-vecchi}. In recent work of Xie, Wu, and Zhang \cite{wu-xie-zhang} it was proved that for uniform matroids both the $Z$-polynomial and its associated $\gamma$-polynomial are ultra log-concave. We take the opportunity to show that there is a general implication between the ultra log-concavity of the  $\gamma$-polynomial of a symmetric polynomial and the ultra log-concavity of the polynomial itself (Theorem~\ref{thm:gamma-ulc}).

Following the notation of Section~\ref{sec:ultra-log-concavity}, let us denote $\mathbb{R}[x,y]_d$ the space of bivariate homogeneous polynomials of degree exactly $d$. We have the following lemma.

\begin{lemma}\label{lemma:gamma}
     The linear map $\alpha:\mathbb{R}[x,y]_{\lfloor d/2\rfloor} \to \mathbb{R}[x,y]_d$ defined by 
     \[x^i y^{\lfloor d/2\rfloor -i} \longmapsto (xy)^i(x+y)^{d-2i}\; \text{ for $i=0,\ldots,\lfloor d/2\rfloor$},\] 
     preserves stability and the Lorentzian property.
\end{lemma}

\begin{proof}
    Suppose $P(x,y) \in \mathbb{R}[x,y]_{\lfloor d/2\rfloor}$ is a Lorentzian (resp. stable) polynomial. Then 
    $$
        \alpha(P)(x,y) = (x+y)^{\lceil d/2\rceil} P\left(\tfrac{xy}{x+y},x+y\right).
    $$
Applying Lemma~\ref{transf} to the first variable, thereby introducing a new variable $z$, followed by change of variables $y\mapsto x+y$ and diagonalization $y=z$, we see that $\alpha(P)(x,y)$ is Lorentzian (resp. stable).
\end{proof}

\begin{theorem}\label{thm:gamma-ulc}
    Let $f(x)$ be a symmetric polynomial with center of symmetry $d/2$. If $\gamma_f(x)\in \ULC\left(\lfloor d/2\rfloor\right)$, then $f(x)\in \ULC(d)$.
\end{theorem}

\begin{proof}
    By assumption, the polynomial $\gamma_f(x)$ is ultra log-concave without internal zeros. This is equivalent to  $$P(x,y) := \sum_{j=0}^{\lfloor {d}/{2}\rfloor} \gamma_i x^iy^{\lfloor d/2\rfloor -i}$$
    being Lorentzian. In particular, using Lemma~\ref{lemma:gamma}, we obtain that $\alpha(P)(x,y)$ is Lorentzian as well, which in turn is equivalent to the ultra log-concavity without internal zeros of $f(x)$ as desired.
\end{proof}

\begin{example}\label{ex:noconversegammaulc}
    There is no reasonable converse for the preceding statement. Precisely, it is generally not true that if $\gamma_f(x)$ has nonnegative coefficients and $f(x)$ is ultra log-concave with no internal zeros, then $\gamma_f(x)$ is ultra log-concave. For instance, consider the ultra log-concave polynomial
    \[ f(x) = x^6 + 8x^5 + 24x^4 + 36x^3 + 24x^2 + 8x + 1.\]
    The reader may check that
    \[ \gamma_f(x) = 2x^3+x^2+2x+1.\]
    Though $f(x)$ is $\gamma$-positive and ultra log-concave, the polynomial $\gamma_f(x)$ is not even unimodal. 
\end{example}

\section{Symmetric decompositions}\label{sec:sym-dec}

\noindent In this section we consider properties of symmetric decompositions and their preservation under Hadamard products. In particular, we will show preservation of nonnegative, interlacing and $\gamma$-positive symmetric decomposition. As a byproduct, we will also show that the Hadamard product preserves the property of having no internal zeros.

As we shall see, the results can be stated and proved in an equivalent way in terms of $f$-polynomials and their diamond products; we begin by giving the necessary background on these concepts which can be found in~\cite{branden,branden-solus}.

In order to study properties of products $\mathscr{W}(pq)$ for polynomials $p,q$ we consider the \emph{subdivision operator}, defined as the bijective linear map $\mathscr{E}:\mathbb{R}[x]\to \mathbb{R}[x]$ defined by $\binom x j \longmapsto x^j$ for each $j \in \NN$. 

To a fixed polynomial $p\in \mathbb{R}[x]$ of degree $d$ we will frequently consider the polynomials $f,g\in \mathbb{R}[x]$ obtained by applying the operators $\mathscr{E}$ and $\mathscr{W}$ to $p$, that is,
    \[ f(x) := \mathscr{E}(p)(x) \qquad \text{ and } \qquad h(x) := \mathscr{W}(p)(x),\]
    and call $\mathscr{E}(p)(x)$ the \emph{$f$-polynomial} of $h=\mathscr{W}(p)(x)$ (with respect to $d=\deg p$), denoted $f=f(h;x)$. The two polynomials $f$ and $h$ are related via the M\"obius transformation
    \begin{equation}\label{eq:f-h-relation} 
    f(x) = (1+x)^{d}\, h\left(\frac{x}{1+x}\right) \, .
    \end{equation}
    Equivalently, if $h=\sum _{i=0}^d h_i t^i$ then $f(x)=\sum _{i=0}^d h_i \, x^i (x+1)^{d-i}$. Via this relation properties of $\mathscr{W}(p)(x)$ can be translated into properties of $\mathscr{E}(p)(x)$. In particular, the Hadamard product can be expressed in terms of the \emph{diamond product} of polynomials. For two polynomials $f,g\in \mathbb{R}[x]$, it is defined as
    \begin{equation}\label{eq:diamond-definition} 
    (f \diamond g)(x):= \sum_{j\geq 0} \frac{f^{(j)}(x)}{j!} \frac{g^{(j)}(x)}{j!}\, x^j\, (x+1)^j,
    \end{equation}
where the notation $f^{(j)}(x)$ stands for the $j$-th order derivative of $f$. By a result of Wagner \cite[Theorem~2.3]{wagner}, it follows that for any two polynomials $p,q$,
    \begin{equation}\label{eq:diamond-e}
        \mathscr{E}(pq) = \mathscr{E} (p)\diamond \mathscr{E} (q) .
    \end{equation} 

    As we shall see next, the notion of symmetric decomposition of a polynomial $h=\mathscr{W}(p)$ can be described via a corresponding decomposition of the polynomial $f=\mathscr{E}(p)$. 

    Recall that for any polynomial $h$ of degree at most $d$ there exist unique symmetric polynomials $a$ and $b$ such that $h(x)=a(x)+x\,b(x)$, $a=\mathcal{I}_d(a)$, and $b(x)=\mathcal{I}_{d-1}(b)$.  The pair $(a,b)$ is called the \emph{symmetric decomposition} or \emph{$\mathcal{I}_d$-decomposition} of $h$ .

    We will consider symmetric decompositions in terms of the $f$-polynomial. Following  Br\"and\'en and Solus~\cite{branden-solus}, for any polynomial $f$ of degree at most $d$ let
\[
\mathscr{R}_d( f )(x)=(-1)^d f(-x-1) \, .
\]
Equivalently, if $f=\sum _{i=0}^d h_i \, x^i (x+1)^{d-i}$, then $\mathscr{R}_d (f)=\sum _{i=0}^d h_{i} \, x^{d-i} (x+1)^{i}$. That is, $\mathscr{R}_d(f(h;x))=f(\mathcal{I}_d(h);x)$. Whenever $d$ is understood from context, we will drop the subscript and write $\mathscr{R}(p)$ instead. The map $\mathscr{R}$ interacts well with the subdivision operator, that is, $\mathscr{E}\circ \mathscr{R} = \mathscr{R}\circ \mathscr{E}$~\cite[Lemma~4.3]{branden-polya}, and consequently also with the diamond product: for all polynomials $f,g$ \[
\mathscr{R}(f\diamond g) = \mathscr{R}(f) \diamond \mathscr{R} (g) \, .
\]
Furthermore, the following two results will be used repeatedly in our proofs. 
\begin{lemma}{{\cite[Lemma~2.2]{branden-solus}}}\label{lem:Rformula}
    Let $f$ be a polynomial of degree at most $d$. Then there exist unique polynomials $\tilde{a}$ and $\tilde{b}$ of degree at most $d$ and $d-1$, respectively, such that $\mathscr{R}_d (\tilde{a})=\tilde{a}$, $\mathscr{R}(\tilde{b})=\tilde{b}$ and
    \[
    f \ = \ \tilde{a} + x\tilde{b} \, .
    \]
    Moreover, 
    \[
\tilde{a}=(x+1)f-x\mathscr{R}_d(f) \quad \text{and} \quad \tilde{b}=\mathscr{R}_d(f)-f \, .\]
\end{lemma}
The pair $(\tilde{a},\tilde{b})$ in the previous lemma is called the \emph{(symmetric) $\mathscr{R}_d$-decomposition} of the polynomial $f$. The next lemma clarifies the relation between the $\mathcal{I}_d$-decomposition and the $\mathcal{R}_d$-decomposition.
\begin{lemma}{{\cite[Lemma~2.3]{branden-solus}}}\label{lem:IdRd}
    Let $h$ be a polynomial of degree at most $d$ and let $f=f(h;x)$ be its $f$-polynomial (with respect to $d$). Let $(a,b)$ be the $\mathcal{I}_d$-decomposition of $h$ and $(\tilde{a},\tilde{b})$ be the $\mathscr{R}_d$-decomposition of $f$. Then
    \[
    \tilde{a}=(x+1)^d a\left( \frac{x}{x+1}\right) \quad \text{and} \quad  \tilde{b}=(x+1)^{d-1} b\left( \frac{x}{x+1}\right)\, .
    \]
    That is, $\tilde{a}=f(a;x)$ and $\tilde{b}=f(b;x)$ are the $f$-polynomials of $a$ and $b$ (with respect to $d$, respectively, $d-1$).
\end{lemma}
\subsection{Nonnegative symmetric decompositions}
The previous lemmas allow us now to translate properties of $\mathscr{W}(p)=h(x)$ into properties of its corresponding $f$-polynomial
$f(h;x)$.

To that end, we will consider $f(h;x)$ expanded in the basis $\{x^i(x+1)^{d-i}\}_{i=0}^d$ which is sometimes referred to as the \emph{magic basis}.
Accordingly, a polynomial $f(x)$ of degree $d$ is called \emph{magic positive} if it can be written as
\[
f(x)=\sum _{i=0}^d h_i \, x^i (x+1)^{d-i},
\]
where $h_0,h_1,\ldots, h_d \geq 0$. In this case we call the set $I=\{i\in [d]\colon h_i>0\}$ the \emph{magic support}, denoted $\msupp(f)$. Notice that by definition we have
 \[ \text{$\mathscr{W}(p)$ has nonnegative coefficients} \iff \text{$\mathscr{E}(p)$ is magic positive}\]
 and $\supp \mathscr{W}(p)=\msupp \mathscr{E}(p)$.
 
The symmetric decomposition $(a,b)$ of $h=\mathcal{W}(p)$ is called \emph{nonnegative} if both $a$ and $b$ have only nonnegative coefficients. An immediate consequence of Lemma~\ref{lem:IdRd} is that $h=\mathscr{W}(p)$ has a nonnegative symmetric decomposition if and only if the $f$-polynomial $f(h;x) =\mathscr{E}(p)$ has a magic positive $\mathscr{R}_d$-decomposition $(\tilde{a},\tilde{b})$, i.e., if both $\tilde{a}$ and $\tilde{b}$ are magic positive. As the following lemma shows, the property of magic positivity is preserved under basic operations. Observe that we recover the preservation of nonnegative coefficients under the Hadamard product which was previously shown in~\cite[Proposition 5.1 (i)]{fischer-kubitzke}.

\begin{lemma}\label{lem:magiclemma}
    Let $f,g$ and $h$ be magic positive. Then
    \begin{enumerate}[\normalfont(i)]
        \item \label{it:magic-i} $f'$ is magic positive.
        \item \label{it:magic-ii} $fg$ is magic positive.
        \item \label{it:magic-iii} $f+g$ is magic positive whenever $\deg f=\deg g$.
        \item \label{it:magic-iv} $f\diamond g$ is magic positive.
                \item \label{it:magic-v} $\msupp (f\diamond h) \subseteq \msupp (g\diamond h)$ whenever $\msupp f \subseteq \msupp g$ and $\deg f=\deg g$.
    \end{enumerate}
    In particular, if $\mathscr{W}(p)$ and $\mathscr{W}(q)$ have nonnegative coefficients then so does $\mathscr{W}(pq)$. In this case, the support of $\mathscr{W}(pq)$ only depends on the support of $\mathscr{W}(p)$ and $\mathscr{W}(q)$.
\end{lemma}

\begin{proof}
    For \ref{it:magic-i} it suffices to prove the claim for $f(x)=x^i\,(x+1)^{d-i}$, for $i=0,\ldots, d$:
    \[
    \frac{d}{dx} (x^i\,(x+1)^{d-i})=ix^{i-1}(x+1)^{d-i}+(d-i)x^i \,(x+1)^{d-i-1}
    \]
    The proofs of \ref{it:magic-ii} and \ref{it:magic-iii} are  immediate, while \ref{it:magic-iv} follows from \ref{it:magic-i}, \ref{it:magic-ii}, \ref{it:magic-iii}, and the definition of the diamond product. In particular, by~\ref{it:magic-iv} and equation~\eqref{eq:diamond-e}, the Hadamard product preserves nonnegative coefficients. 

    For~\ref{it:magic-v} we observe that since $\msupp f$ is contained in $\msupp g$ and both $f$ and $g$ are magic positive, there exists $\lambda >0$ and a magic positive polynomial $p$ of degree $\deg p=\deg f$ such that $f+p=\lambda g$. It follows that
    \[
    \msupp (g\diamond h)=\msupp (\lambda (g\diamond h))=\msupp (f\diamond h)\cup \msupp (p\diamond h)
    \]
    where the last step follows from the bilinearity of the diamond product and since both $f\diamond g$ and $p\diamond h$ are magic positive by~\ref{it:magic-iv}. This proves~\ref{it:magic-v}. It follows that if $\msupp f=\msupp g$ then also $\msupp (f\diamond h)=\msupp (f\diamond h)$. That is, the magic support of the diamond product only depends on the magic support of its factors (and not on the actual polynomials), and thus the same is true for the support of the Hadamard product.
\end{proof}

From the preceding lemma it follows that the property of having no internal zeros is preserved by the Hadamard product.

\begin{proposition}\label{prop:no-int-zeros}
        Let $p,q$ be polynomials such that both $\mathscr{W}(p)$ and $\mathscr{W}(q)$ have only nonnegative coefficients and no internal zeros. Then $\mathscr{W}(pq)$ has also no internal zeros.
\end{proposition}
\begin{proof}
    Since $\mathscr{W}(p)$ and $\mathscr{W}(q)$ have no internal zeros there are indices $0\leq i_0, i_1, j_0, j_1$ such that 
    \begin{eqnarray*}
        \supp \mathscr{W}(p) &=& \supp ( x^{i_0}(x+1)^{i_1})\\
        \supp \mathscr{W}(q) & = & \supp (x^{j_0}(x+1)^{j_1}) \, .
    \end{eqnarray*}
    By Lemma~\ref{lem:magiclemma}, the support of the Hadamard product $\mathscr{W}(pq)$ is equal to the support of the Hadamard product of $x^{i_0}(x+1)^{i_1}$ and $x^{j_0}(x+1)^{j_1}$.
    As a consequence of Theorem~\ref{thm:wagner-main}, the latter has no internal zeros.
\end{proof}

Next, we show that the property of having a nonnegative symmetric decomposition is preserved by the Hadamard product. For that we also employ the following lemma.

\begin{lemma}\label{lem:defect1}
    Suppose $\tilde{b}_1, \tilde{b}_2$ are polynomials of degree $d_1-1, d_2-2$ such that $\mathscr{R}_{d_1-1}(\tilde{b}_1)=\tilde{b}_1$ and $\mathscr{R}_{d_2-1}(\tilde{b}_2)=\tilde{b}_2$. 
    
    Then, there exists a unique polynomial $\ell(x)$ such that $\mathscr{R}_{d_1+d_2-1}(\ell) =\ell$ and 
    \begin{eqnarray*}
        (x+1)\tilde{b}_1\diamond (x+1)\tilde{b}_2 &=& (x+1)\ell (x) \, , \quad \text{and}\\
        x\tilde{b}_1\diamond x\tilde{b}_2 &=& x\ell (x)\, .
        \end{eqnarray*}
        Moreover if $\tilde{b}_1$ and $\tilde{b}_2$ are magic positive, then so is $\ell(x)$.
\end{lemma}
\begin{proof}
    Let $p_i$ be a polynomial of degree $d_i$ such that $\mathscr{E}(p_i)=(x+1)\tilde{b}_i$, for $i=1,2$. Write 
\[
        \mathscr{W}(p_i) = \sum _{j=0}^{d_i} h_j^i x^j\, , \quad i=1,2 \, .
\]
Then $(x+1)\tilde{b}_i=\sum _{j=0}^{d_i} h_j^i x^j(x+1)^{d_i-j}$. Hence $h_{d_i}^i=0$ and 
\[
\tilde{b}_i = \sum _{j=0}^{d_i-1} h_j^i x^j(x+1)^{d_i-1-j}\, .
\]
Since $\mathscr{R}_{d_i-1}(\tilde{b}_i)=\tilde{b}_i$, it follows that  $h_j^i=h_{d_i-1-j}^i$ for all $0\leq j\leq d_i-1$. That is, $\mathscr{W}(p_i)$ is symmetric with defect $\defect \mathscr{W}(p_i)=1$ for $i=1,2$. By Theorem~\ref{thm:gamma-positivity}, also $\mathscr{W}(p_1p_2)$ is symmetric with defect $\defect \mathscr{W}(p_1p_2)=1$. Thus, the corresponding $f$-polynomial is of the form 
\[
\mathscr{E}(p_1p_2)=(x+1)\tilde{b}_1\diamond (x+1)\tilde{b}_2=(x+1)\sum _{j=0}^{d_1+d_2-1}c_jx^j(x+1)^{d_1+d_2-1-j}
\]
with $c_j=c_{d_1+d_2-1-j}$ for $0\leq j \leq d_1+d_2-1$. Hence the first claim follows by setting $\ell = \sum _{j=0}^{d_1+d_2-1}c_jx^j(x+1)^{d_1+d_2-1-j}$. For the second claim we observe that
\[
x\ell = \mathscr{R}((x+1)\ell) = \mathscr{R}\left((x+1)\tilde{b}_1\right)\diamond \mathscr{R}\left((x+1)\tilde{b}_2\right) = x\tilde{b}_1\diamond x\tilde{b}_2
\]
where we used the compatibility of $\mathscr{R}$ with the diamond product and the $\mathscr{R}$-symmetry of $\tilde{b}_1,\tilde{b}_2$ and $\ell$. 

Moreover if $\tilde{b}_1$ and $\tilde{b}_2$ are magic positive,  then so is $x\tilde{b}_1\diamond x\tilde{b}_2=x\ell $ by Lemma~\ref{lem:magiclemma}. Hence so is $\ell$.
\end{proof}

\begin{proposition}\label{prop:nonnegativesymdecomp}
    Let $p,q$ be polynomials such that $\mathscr{W}(p)$ and $\mathscr{W}(q)$ have nonnegative symmetric $\mathcal{I}_{d_1}$- and $\mathcal{I}_{d_2}$-decompositions, respectively. Then $\mathscr{W}(pq)$ has a nonnegative symmetric $\mathcal{I}_{d_1+d_2}$-decomposition. 
\end{proposition}
\begin{proof}
    We will show equivalently that if the $f$-polynomials $f_1,f_2$ have magic positive $\mathscr{R}_{d_1}$- and $\mathscr{R}_{d_2}$-decompositions, then $f_1\diamond f_2$ has a magic positive $\mathscr{R}_{d_1+d_2}$-decomposition. Let $(\tilde{a}_i,\tilde{b}_i)$ be the decomposition of $f_i$, $i=1,2$, and let $(\tilde{c},\tilde{d})$ be the decomposition of $f_1\diamond f_2$. Then by Lemma~\ref{lem:Rformula}
    \begin{eqnarray*}
    \tilde{d} &=& \mathscr{R}(f_1\diamond f_2)-f_1\diamond f_2 \\
    &=& \mathscr{R}(f_1) \diamond \mathscr{R}(f_2)-f_1\diamond f_2\\
    &=& (\tilde{a}_1+(x+1)\tilde{b}_1)\diamond (\tilde{a}_2+(x+1)\tilde{b}_2)-(\tilde{a}_1+x\tilde{b}_1)\diamond (\tilde{a}_2+x\tilde{b}_2)\\
    &=&  \tilde{b}_1\diamond \tilde{a}_2+ \tilde{a}_1 \diamond \tilde{b}_2 + ((x+1)\tilde{b}_1)\diamond ((x+1)\tilde{b}_2)-(x\tilde{b}_1)\diamond (x\tilde{b}_2) \, .\end{eqnarray*}
    The first two summands are magic positive of degree $d_1+d_2-1$ by Lemma~\ref{lem:magiclemma}. For the difference of the last summands, we observe that it is equal to a magic positive polynomial of degree $d_1+d_2-1$ by Lemma~\ref{lem:defect1}.
    Thus, $\tilde{d}$ is a sum of three magic positive terms of the same degree and thus itself magic positive.
        
    Analogously, for $\widetilde{c}$ we have:
    \begin{align*}
        \tilde{c}\enspace &= \enspace \tilde{f}_1\diamond \tilde{f}_2 - x \tilde{d} \\
        &=\enspace \tilde{a}_1\diamond \tilde{a}_2+(x\tilde{b}_1)\diamond \tilde{a}_2+\tilde{a}_1\diamond (x\tilde{b}_2)+(x\tilde{b}_1)\diamond (x\tilde{b}_2)\\
        &\qquad  -x\left(\tilde{b}_1\diamond \tilde{a}_2+\tilde{a}_1\diamond \tilde{b}_2+((x+1)\tilde{b}_1)\diamond ((x+1)\tilde{b}_2)-(x\tilde{b}_1)\diamond (x\tilde{b}_2)\right)\\
        &=\enspace\tilde{a}_1\diamond \tilde{a}_2+
        \left[\tilde{a}_1\diamond (x\tilde{b}_2)-x(\tilde{a}_1\diamond \tilde{b}_2)\right]+\left[(x\tilde{b}_1)\diamond \tilde{a}_2-x(\tilde{b}_1\diamond \tilde{a}_2)\right]\\
        &\qquad  +\left[(x\tilde{b}_1)\diamond (x\tilde{b}_2)-x\left(((x+1)\tilde{b}_1)\diamond ((x+1)\tilde{b}_2)-(x\tilde{b}_1)\diamond (x\tilde{b}_2)\right)\right]
    \end{align*}
    We will show that all four terms in the equation just above  are magic positive of the same degree $d_1+d_2$ (or zero). For the first one, that is, $\tilde{a}_1\diamond \tilde{a}_2$, this claim is immediate by Lemma~\ref{lem:magiclemma}. For the second term, we observe that
    \[
    \frac{d^k}{dx^k}(x\tilde{b}_2(x))=x\tilde{b}_2^{(k)}+k\tilde{b}_2^{(k-1)}
    \]
    and thus, by the definition of the diamond product,
    \[
    \tilde{a}_1\diamond x\tilde{b}_2=x(\tilde{a}_1\diamond \tilde{b}_2)+\sum _{k\geq 1}\frac{\tilde{a}_1^{(k)}}{k!}\frac{\tilde{b}_2^{(k-1)}}{(k-1)!}x^k(x+1)^{k}
    \]
    where the sum appearing on the right is itself magic positive because each term is (Lemma~\ref{lem:magiclemma}). Notice that differentiating and multiplying a magic positive polynomial by $x$ or $x+1$ preserves this property. This shows that the second term is magic positive, and similarly, one shows that the third term is. For the last summand, we again apply Lemma~\ref{lem:defect1} and have $((x+1)\tilde{b}_1)\diamond ((x+1)\tilde{b}_2)=(x+1)\ell (x)$ for a magic positive polynomial $\ell =\mathscr{R}(\ell)$. It follows that
    \[
    (x\tilde{b}_1)\diamond (x\tilde{b}_2)-x\left(((x+1)\tilde{b}_1)\diamond ((x+1)\tilde{b}_2)-(x\tilde{b}_1)\diamond (x\tilde{b}_2)\right)=x\ell -x\left( (x+1)\ell -x\ell\right)=0
    \]
    which is clearly also magic positive. This completes the proof.
\end{proof}

\subsection{Interlacing symmetric decompositions}


Our next goal is to show that nonnegative, interlacing symmetric decompositions are preserved under Hadamard products. Recall, that given two real-rooted polynomials
 $a$ and $b$ with zeros $s_k\leq \cdots \leq s_1$ and $t_m\leq \cdots \leq t_1$, respectively, we say that $b$ \emph{interlaces} $a$, and write $b\preceq a$, if
    \[ \cdots \leq t_2 \leq s_2\leq t_1 \leq s_1.\]
    In particular, we must have $\deg b \leq \deg a\leq \deg b +1$. The following is a collection of basic properties of interlacing polynomials, see, e.g.,~\cite{wagner}.
\begin{lemma}{{\cite[Section 3]{wagner}}}\label{lem:basicinterlacing}
    Let $f,g,h$ be real-rooted polynomials with nonnegative coefficients. Then
    \begin{enumerate}[\normalfont(i)]
        \item $f \preceq g$ if and only if $g\preceq xf$
        \item if $f\preceq h$ and $g\preceq h$ then $f+g \preceq h$
        \item if $f\preceq g$ and $f\preceq h$ then $f\preceq g+h$
    \end{enumerate}
\end{lemma}

Following~\cite{branden-solus} we call the symmetric decomposition $(a,b)$ of a polynomial $h=\mathscr{W}(p)$ \emph{nonnegative} if both $a$ and $b$ have only nonnegative coefficients, and \emph{interlacing} if $b\preceq a$. Instead of working with the numerator polyomials directly, we will again consider the corresponding $f$-polynomial, $f(h;x)=\mathscr{E}(p)$. Crucially, by equation~\eqref{eq:f-h-relation}, $\mathscr{W}(p)$ and $\mathscr{E}(p)$ are related by a M\"obius transformation that preserves real-rootedness of polynomials. In particular, $h$ has zeros in $(-\infty,0]$, i.e. only nonnegative coefficients,  if and only if $f$ has zeros in $(-1,0]$. Moreover, the transformation preserves interlacing polynomials with zeros in these intervals. It follows that if $(\tilde{a},\tilde{b})$ is the $\mathscr{R}$-decomposition of $f(h;x)$ and the $\mathcal{I}_d$ composition $(a,b)$ of $h$ has only nonnegative coefficients then 
\[
a \ \preceq \ b \quad \text{if and only if} \quad \tilde{a} \ \preceq \ \tilde{b} \, .
\]
In this case, we say that $f$ has an interlacing $\mathscr{R}$-decomposition. Furthermore, by~\cite[Theorem~2.6]{branden-solus} this is equivalent to 
\[
\mathscr{R}(f) \ \preceq \ f  \, .
\]
We will use this equivalence in the proof of Theorem~\ref{thm:interlacingsymdec} below. Furthermore, we will employ the fact that the interlacing property behaves well with the diamond product, as well as a basic lemma on interlacing sequences of polynomials.
\begin{theorem}{{\cite[Theorem 12]{BrandenOperators}}}\label{thm:diamondinterlacing}
    Let $f,g$ and $h$ be real-rooted polynomials, and let $h$ have zeros in $[-1,0]$. Then $f\diamond h$ is real-rooted, and if $f\preceq g$ then
    \[
    f\diamond h \ \preceq \ g \diamond h \, .
    \]
\end{theorem}
\begin{lemma}{{\cite[Lemma 2.3]{branden-polya}}}\label{lem:endsinterlacing}
    If $f_0,f_1,\ldots, f_m$ are real rooted polynomials with $f_{i-1}\preceq f_{i}$ for $1\leq i\leq m$ and $f_0\preceq f_m$, then $f_i\preceq f_j$ whenever $i\leq j$.
\end{lemma}
A sequence of polynomials as in Lemma \ref{lem:endsinterlacing} is called an \emph{interlacing sequence}. 

We are now ready to prove our main result of this section.
\begin{proof}[Proof of Theorem~\ref{thm:interlacingsymdec}]
    We prove the equivalent statement for $f$-polynomials. We may thus assume that we have $f$-polynomials $f_1,f_2$ of degrees $d_1$ and $d_2$ with magic positive $\mathscr{R}_{d_1}$- and $\mathscr{R}_{d_2}$-decompositions $(\tilde{a}_i,\tilde{b}_i)$ , $i=1,2$, and $\tilde{b}_i\preceq \tilde{a}_i$. Let $(\tilde{c},\tilde{d})$ be the $\mathscr{R}_{d_1+d_2}$-decomposition of $f_1\diamond f_2$. We will show that $\mathscr{R}(f_1\diamond f_2)\preceq f_1\diamond f_2$.

    Since $\tilde{b}_i\preceq \tilde{a}_i$, $\deg \tilde{b}_i=\deg \tilde{a}_i -1$, and the zeros of $\tilde{b}_i,\tilde{a}_i$ lie in $[-1,0]$, for $i=1,2$, by Theorem~\ref{thm:diamondinterlacing} we have the following chain of interlacing polynomials.
    \[
    ((x+1)\tilde{b}_1)\diamond((x+1)\tilde{b}_2)\preceq ((x+1)\tilde{b}_1)\diamond\tilde{a}_2 \preceq \tilde{a}_1\diamond\tilde{a}_2\preceq (x\tilde{b}_1)\diamond\tilde{a}_2\preceq (x\tilde{b}_1)\diamond (x\tilde{b}_2)
    \]
    Observe that ends are interlacing by Lemma~\ref{lem:defect1}. By Lemma~\ref{lem:endsinterlacing} the above sequence is interlacing. 
    
        Analogously, the sequence 
        \[
    ((x+1)\tilde{b}_1)\diamond((x+1)\tilde{b}_2)\preceq \tilde{a}_1\diamond((x+1)\tilde{b}_2) \preceq \tilde{a}_1\diamond\tilde{a}_2\preceq \tilde{a}_1\diamond(x\tilde{b}_2)\preceq (x\tilde{b}_1)\diamond (x\tilde{b}_2)
    \]
    is interlacing. In particular $\tilde{a}_1\diamond \tilde{a}_2\preceq (x\tilde{b}_1)\diamond (x\tilde{b}_2)$.
    By Lemma~\ref{lem:basicinterlacing}, 
    \begin{equation}\label{eq:interlacingright}
            \tilde{a}_1\diamond \tilde{a}_2\preceq f_1 \diamond f_2=\tilde{a}_1\diamond \tilde{a}_2+\tilde{a}_1\diamond(x\tilde{b}_2)+(x\tilde{b}_1)\diamond\tilde{a}_2+(x\tilde{b}_1)\diamond (x\tilde{b}_2) \preceq (x\tilde{b}_1)\diamond (x\tilde{b}_2)
    \end{equation}
    Further, observe that
    \[
    \mathscr{R}(f_1\diamond f_2) = \mathscr{R}(f_1)\diamond \mathscr{R}(f_2) = (\tilde{a}_1+(x+1)\tilde{b_1})\diamond (\tilde{a}_2+(x+1)\tilde{b}_2) \, .
    \]
By using the interlacing sequences above and applying Lemma~\ref{lem:basicinterlacing} we deduce 
    \begin{equation}    \label{eq:interlacingleft}
            ((x+1)\tilde{b}_1)\diamond((x+1)\tilde{b}_2)\preceq \mathscr{R}(f_1\diamond f_2) \preceq    \tilde{a}_1\diamond \tilde{a}_2 \, .
    \end{equation}
    Combining~\eqref{eq:interlacingright} and~\eqref{eq:interlacingleft} together with Lemma~\ref{lem:endsinterlacing} yields $\mathscr{R}(f_1\diamond f_2)\preceq f_1\diamond f_2$ as claimed.
\end{proof}

\begin{example}\label{example:real-rootedness-sym-fails}
    It is not true that if one just requires the symmetric decompositions of $\mathscr{W}(p)$ and $\mathscr{W}(q)$ to be real-rooted (but not interlacing), then $\mathscr{W}(pq)$ is real-rooted. For example, fix any integer $\lambda$ and consider the polynomial $h(x) = x^3+(3+\lambda)x^2+3x+1$. Its symmetric decomposition is given by $a(x)=1+3x+3x^2+x^3$ and $b(x) = \lambda x$, both real-rooted polynomials (however, notice that $h(x)$ itself is not real-rooted whenever $\lambda\neq 0$). Let us call $p(x)$ the only cubic polynomial such that $(\mathscr{W}(p))(x) = h(x)$. For $\lambda=6$, one may compute $(\mathscr{W}(p^2))(x) = x^{6} + 162 \, x^{5} + 1239 \, x^{4} + 1836 \, x^{3} + 639 \, x^{2} + 42 \, x + 1$, which is not real-rooted and whose symmetric decomposition is not real-rooted.
\end{example}

If the $\mathcal{I}_d$-decomposition $(a,b)$ of a polynomial $h$ satisfies that both $a$ and $b$ are $\gamma$-positive, we say that $h$ has a $\gamma$-positive symmetric decomposition. We can use the previous result to show the preservation of the $\gamma$-positivity of symmetric decompositions under Hadamard products.

\begin{proposition}\label{prop:gamma-symmetric}
 Let $\mathscr{W}(p)$ and $\mathscr{W}(q)$ have $\gamma$-positive symmetric decompositions. Then so does $\mathscr{W}(pq)$.   
\end{proposition}

\begin{proof}
    Every polynomial $h$ admitting a $\gamma$-positive $\mathcal{I}_d$-symmetric decomposition can be written uniquely as a nonnegative linear combination of polynomials of the form
    $x^i(x+1)^{d-2i}$ and $x\cdot x^k (x+1)^{d-1-2k}$. By bilinearity it thus suffices to proof the claim if both $\mathscr{W}(p)$ and $\mathscr{W}(q)$ are given by such basis elements. In this case both $\mathscr{W}(p)$ and $\mathscr{W}(q)$ trivially have a nonnegative and interlacing decomposition. Thus by the previous theorem the same is true for $\mathscr{W}(pq)$. In particular, both parts of the symmetric decomposition of $\mathscr{W}(pq)$ are symmetric and real-rooted polynomials, and thus are $\gamma$-positive.
\end{proof}

\begin{remark}
    We point out that the preceding result does not imply Theorem~\ref{thm:gamma-positivity} under the assumption that $\mathscr{W}(p)$ and $\mathscr{W}(q)$ are themselves symmetric polynomials. Notice that in Theorem~\ref{thm:gamma-positivity} we considered polynomials with \emph{arbitrary} center of symmetry, whereas for Proposition~\ref{prop:gamma-symmetric} we require that the center of symmetry equals $d/2$.
\end{remark}
The preceeding result also completes the proof of Theorem~\ref{thm:nonnegativesymdec}.
\begin{proof}[Proof of Theorem~\ref{thm:nonnegativesymdec}]
    The result follows from combining Propositions~\ref{prop:nonnegativesymdecomp} and~\ref{prop:gamma-symmetric}.
\end{proof}

\section{Hadamard powers and their behavior in the limit}\label{sec:Hadamardpowers}

\noindent The goal of this section is to disprove Conjecture~\ref{conj:fischer-kubitzke} posed by Fischer and Kubitzke \cite{fischer-kubitzke}. We will disprove it in a strong sense by showing that even weaker statements on log-concavity fail to be true.

It follows from the Newton inequalities that any polynomial having only nonpositive real zeros is log-concave. Furthermore, given $h=\mathscr{W}(p)$ and $f=\mathscr{E}(p)$ for some polynomial $p\in \mathbb{R}[x]$ such that $h$ has only nonnegative, log-concave coefficients without internal zeros, it follows that also $f$ is log-concave without internal zeros by a result of Brenti~\cite[Theorem~2.5.8]{brenti}. We will use this implication to give a counterexample to Conjecture~\ref{conj:fischer-kubitzke} by considering the $f$-polynomial corresponding to the Ehrhart polynomial of iterated products of a certain polytope.

Recall that for every lattice polytope $P$, the $k$-fold Cartesian product of $P\subseteq \mathbb{R}^d$ with itself is
    \[P^k := \underbrace{P\times \cdots \times P}_{\text{$k$ times}} = \{(p_1,\ldots, p_k)\colon p_i \in P\}\subseteq \left(\mathbb{R}^d\right)^k \, .\]

    Clearly, for all integers $n\geq 0$, one has $|nP^k \cap \mathbb{Z}^{dk}|=|nP \cap \mathbb{Z}^{d}|^k$. Therefore, $E_{P^k}(x) = E_P(x)^k$ and, consequentially, by equation~\eqref{eq:diamond-e},
    \[ f_{P^k}(x) = f^{\diamond k}_P(x) := (\underbrace{f_P \diamond\cdots \diamond f_P}_{\text{$k$ times}})(x),\]
where $E_P(x)$ denotes the Ehrhart polynomial and $f_P = \mathscr{E}(E_P)$ the corresponding $f$-polynomial of the polytope $P$.

We consider the $3$-dimensional Reeve tetrahedron with vertices $(0,0,0)$, $(1,0,0)$, $(0,1,0)$ and $(1,7,8)$ and show that the $f$-polynomial of the Ehrhart polynomial of all powers of it fail to be log-concave, hence refuting Conjecture~\ref{conj:fischer-kubitzke}. 

\begin{theorem}\label{thm:Ehrhartcounterexample}
    Let $P\subseteq \mathbb{R}^3$ be the $3$-dimensional Reeve tetrahedron with vertices $(0,0,0)$, $(1,0,0)$, $(0,1,0)$ and $(1,7,8)$. Let $f_{P^k}(x) = f_{k,0} + f_{k,1} x +\cdots + f_{k,3k}x^{3k}$. Then
        \[ f_{k,1}^2 < f_{k,0} f_{k,2}\]
        for all $k\geq 1$.
    In particular, for all $k\geq 1$, the $h^\ast$-polynomial of the $k$-fold product of $P$ with itself fails to be log-concave and, consequentially, is not real-rooted.
\end{theorem}
\begin{proof}
    The $h^\ast$-polynomial of this tetrahedron is known to be 
    \[
    h^\ast _P (x) \ = \ 1+7x^2,
    \]
    see, e.g., \cite[Example~2.4]{ferroni-higashitani}.
    It follows that the corresponding $f$-polynomial equals
    \[
        f_P(t) \ = \ (x+1)^3h^\ast _P \left(\frac{x}{x+1}\right) \ = \  8x^3 + 10x^2 + 3x+1 \, .
    \]
    By definition, we can calculate $f_P^{\diamond (k+1)}(x)$ as follows:
    \begin{align*}
    (f_P \diamond f_P^{\diamond k})(x) & =  \sum_{i=0}^3 \frac{f_P^{(i)}(x)}{i!}\frac{(f_P^{\diamond k})^{(i)}(x)}{i!}\, x^i\, (x+1)^i \\
    &= (1+3x+10x^2+8x^3)\left(\sum_{\ell \geq 0}f_{k,\ell}\,x^\ell\right)\\
    &\enspace + (3+20x+24x^2)\left(\sum _{\ell \geq 0}(\ell+1)f_{k,\ell +1}\, x^\ell\right)\,x\,(x+1)\\
    &\enspace + \frac{1}{4}(20+48x)\left(\sum _{\ell \geq 0}(\ell+2)(\ell +1)f_{k,\ell +2}\, x^\ell\right)\, x^2\, (x+1)^2\\
    &\enspace +\frac{1}{36}\cdot 48 \cdot \left(\sum _{\ell \geq 0}(\ell+3)(\ell+2)(\ell +1)f_{k,\ell +3}\, x^\ell \right)\, x^3\, (x+1)^3  .
    \end{align*}
    Comparing the first three coefficients on both sides, we obtain the following recursion:
    \begin{align*}
        f_{k+1,0} & =  f_{k,0}, \\
        f_{k+1,1} & =  3f_{k,0}+4f_{k,1},\\
        f_{k+1,2} & =  10f_{k,0}+26f_{k,1} + 17 f_{k,2} \, .
    \end{align*}
    This allows us to derive a closed formula for these coefficients by induction: 
    \[
        f_{k,0} = 1, \qquad
        f_{k,1} = 4^k-1, \qquad \text{ and } \enspace
        f_{k,2} = 17^k-2\cdot 4^k+1 \, .
    \]
    For $f_{k,0}$ the proof is immediate. For the linear term $f_{k,1}$, notice that when $k=1$, we have $f_{1,1} = 3 = 4^1-1$, and the inductive step is verified by $f_{k+1,1} = 3f_{k,0}+4f_{k,1} = 3 + 4(4^k-1) = 4^{k+1}-1$. For the quadratic term $f_{k,2}$, when $k=1$ we have $f_{1,2} = 10 = 17^1-2\cdot 4^1 + 1$, and the induction follows from
    \begin{align*}
        f_{k+1,2}  &= 10 f_{k,0}+26f_{k,1} + 17 f_{k,2}\\
     &= 10+26 (4^k-1) + 17 (17^k-2 \cdot 4^k+1)\\
     &= 17^{k+1}-8\cdot 4^k+1\\
     &= 17^{k+1}-2\cdot 4^{k+1}+1 \, .
    \end{align*}
    Thus, we have established the first three coefficients. Observe that this implies that
        \[ f_{k,1}^2 = (4^k-1)^2 = 16^k - 2\cdot 4^k + 1 < 17^k - 2\cdot 4^k +1 = f_{k,2}f_{k,0}, \]
    and therefore none of the polynomials $f_{P^k}(x)$ for $k\geq 1$ is log-concave. Consequentially, by the discussion above, we conclude that the $h^*$-polynomials of the polytopes $P^k$ for $k\geq 1$ fail to be log-concave and thus also fail to be real-rooted.
\end{proof}

\begin{proof}[Proof of Theorem~\ref{thm:counterexample}]
    The example given in Theorem~\ref{thm:Ehrhartcounterexample} constitutes a counterexample to Conjecture~\ref{conj:fischer-kubitzke}.
\end{proof}

\section{Final remarks}\label{sec:Final remarks}

\noindent One outstanding question that remains unsolved is the following.

\begin{question}\label{question:log-concavity}
    Assume that $\mathscr{W}(p)$ and $\mathscr{W}(q)$ are polynomials with log-concave coefficients and no internal zeros. Does the same property hold for $\mathscr{W}(pq)$?
\end{question}

This question was considered before by Ferroni and Higashitani \cite[Question~3.4(i)]{ferroni-higashitani} in the context of Ehrhart polynomials. By an example provided in that article \cite[Example~3.5]{ferroni-higashitani}, it is known that if $\mathscr{W}(p)$ and $\mathscr{W}(q)$ are nonnegative and unimodal, it is not necessarily true that $\mathscr{W}(pq)$ is unimodal too. Recently, Balletti \cite{balletti} showed that counterexamples coming from Ehrhart polynomials of lattice polytopes also exist in very large dimensions.

Consider the linear operator $N:\mathbb{R}[x_1,\ldots,x_n]\to \mathbb{R}[x_1,\ldots,x_n]$ defined on monomials by $x^{\alpha} \mapsto x^{\alpha}/\alpha!$, where $\alpha! := \alpha_1 !\cdots \alpha_n!$. A homogeneous polynomial $f\in \H_n^d$ is said to be \emph{denormalized Lorentzian} if the polynomial $N(f)$ is Lorentzian. Notice that we have the following equivalence:
    \begin{align*}
        p\in \H_2^n \text{ is denormalized Lorentzian}  \iff & p(x,1) \text{ is log-concave without internal zeros.}
    \end{align*}

The class of denormalized Lorentzian polynomials is closed under products, and every Lorentzian polynomial is denormalized Lorentzian.

One can prove that the operation described in Lemma~\ref{transf} preserves the property of being denormalized Lorentzian. Using the notation of Lemma~\ref{form}, it follows that if $p$ is Lorentzian and $q$ is denormalized Lorentzian, then $p\bullet q$ is denormalized Lorentzian. In particular, the same strategy employed in the proof of Theorem~\ref{thm:ultra-log-concavity} can be used to show the following statement, which offers some evidence towards a positive answer to Question~\ref{question:log-concavity}.

\begin{proposition}
    If $\mathscr{W}(p)\in \ULC(\deg p)$ and $\mathscr{W}(q)$ is log-concave with no internal zeros, then $\mathscr{W}(pq)$ is log-concave with no internal zeros.
\end{proposition}

Notice, however, that a necessary step in the proof of Theorem~\ref{thm:ultra-log-concavity} was that if $p(x,y)$ is Lorentzian, then $p(x+z,y+z)$ is Lorentzian. An analogous statement does not hold for denormalized Lorentzian polynomials. For example, the polynomial $p(x,y) = x^2+5xy+25y^2$ is denormalized Lorentzian, but $p(x+z,y+z)$ is not denormalized Lorentzian. In particular, there is no straightforward modification of our argument yielding a positive answer to Question~\ref{question:log-concavity}. On the other hand, it would also be of interest to prove that the operator defined in Lemma~\ref{lemma:gamma} preserves the denormalized Lorentzian property. A proof of that fact would imply that if $\gamma_f$ is log-concave with no internal zeros, then $f$ is log-concave with no internal zeros.\footnote{After the posting of the present paper on the arXiv, Ferroni, Panova, and Venturello \cite{ferroni-panova-venturello} proved this implication.}\\

\emph{Acknowledgements:} PB is a Wallenberg Academy Scholar supported by the Knut and Alice Wallenberg Foundation, and the G\"oran Gustafsson foundation. LF was partially supported by grant 2018-03968 of the Swedish Research Council, and by the Minerva Research Foundation at the Institute for Advanced Study. KJ was partially supported by the Wallenberg AI, Autonomous Systems and Software Program funded by the Knut and Alice Wallenberg Foundation as well as grants 2018-03968 and 2023-04063 of the Swedish Research Council, and the Verg Foundation.

\bibliographystyle{amsalpha}
\bibliography{bibliography}

\end{document}